\documentclass{amsart}
\usepackage{amsmath,amsfonts,amsthm}
\usepackage{amssymb,latexsym}
\usepackage{graphics}
\usepackage[colorlinks]{hyperref}

\usepackage[all]{xypic}

\theoremstyle{plain}
\theoremstyle{definition}
\newtheorem{theorem}{Theorem}[section]
\newtheorem{lemma}[theorem]{Lemma}

\newtheorem{example}[theorem]{Example}

\newtheorem{note}[theorem]{Note}

\newtheorem{convention}[theorem]{Convention}

\theoremstyle{remark}

\numberwithin{equation}{section}
\newcommand{\SP}{\: \: \: \: \:}

\title{On generalized shift transformation semigroups}
\author[F. Ayatollah Zadeh Shirazi, F. Ebrahimifar]{Fatemah Ayatollah Zadeh Shirazi, Fatemeh Ebrahimifar}
\begin{document}
\begin{abstract}
In the following text we prove that for finite discrete $X$ with at least two elements and infinite $\Gamma$, 
the generalized shift transformation semigroup $({\mathcal S},X^\Gamma)$ is equicontinuous (resp.  has at 
least an equicontinuous point, is not sensitive) if and only if for all $w\in\Gamma$, $\{\varphi(w):\sigma_\varphi\in{\mathcal S}\}$ 
is finite. We continue our study regarding distality and expansivity of $({\mathcal S},X^\Gamma)$.
\end{abstract}
\maketitle
\noindent {\small {\bf 2010 Mathematics Subject Classification:}  54H20, 37B05 \\
{\bf Keywords:}} distal, equicontinuous, expansive, generalized shift, sensitive, transformation semigroup, weakly almost periodic.
\section{Introduction}
\noindent The concept of generalized shifts has been introduced for the first time in \cite{2008} 
as a generalization of one-sided shift 
$\mathop{\{1,\ldots,k\}^{\mathbb N}\to\{1,\ldots,k\}^{\mathbb N}}\limits_{(a_1,a_2,\cdots)\mapsto(a_2,a_3,\cdots)}$ and two-sided shift 
$\mathop{\{1,\ldots,k\}^{\mathbb Z}\to\{1,\ldots,k\}^{\mathbb Z}}\limits_{(a_n)_{n\in{\mathbb Z}}\mapsto(a_{n+1})_{n\in{\mathbb Z}}}$, 
which are well-known in dynamical systems' approach and ergodic theory \cite{1982}. 
Suppose $K$ is a nonempty set with at least two elements, $\Gamma$ is a nonempty set, 
and $\varphi:\Gamma\to\Gamma$ is an arbitrary map, then 
$\sigma_\varphi:\mathop{K^\Gamma\to K^\Gamma\SP\:}\limits_{(x_\alpha)_{\alpha\in\Gamma}\mapsto(x_{\varphi(\alpha)})_{\alpha\in\Gamma}}$ 
denotes a generalized shift. It's evident that for topological space $K$, $\sigma_\varphi: K^\Gamma\to K^\Gamma$ 
is continuous, where $K^\Gamma$ is equipped with product topology, moreover if $K$ has a group structure, then 
$\sigma_\varphi: K^\Gamma\to K^\Gamma$ is a group homomorphism. 
Dynamical (and non-dynamical) properties of generalized shifts has been studied in several papers, 
like \cite{2013} and \cite{2010}. In the following text we study equicontinuity and distality in 
$({\mathcal S},X^\Gamma)$, where $X$ is a finite discrete space and ${\mathcal S}$ is a semigroup of 
generalized shifts on $X^\Gamma$.
\section{Priliminaries}
\subsection*{Background on uniform spaces} Let's recall that for arbitrary set $Y$ we say ${\mathcal K}$ is a 
{\it uniformity} on $Y$ if $\mathcal K$ is a collection of subsets of $Y\times Y$ such that:
\begin{itemize}
\item $\forall\alpha\in{\mathcal K}\SP(\Delta_Y\subseteq\alpha)$;
\item $\forall\alpha,\beta\in{\mathcal K}\SP(\alpha\cap\beta\in{\mathcal K})$;
\item $\forall\alpha\in{\mathcal K}\SP\forall\beta\subseteq Y\times Y\SP(\alpha\subseteq\beta\Rightarrow\beta\in{\mathcal K})$;
\item $\forall\alpha\in{\mathcal K}\SP\exists\beta\in{\mathcal K}\SP(\beta\circ\beta^{-1}\subseteq\alpha)$;
\end{itemize}
where $\Delta_Y=\{(y,y):y\in Y\}$ and for $\alpha,\beta\subseteq Y\times Y$ we have 
$\alpha^{-1}=\{(y,x): (x,y)\in\alpha\}$, $\alpha\circ\beta=\{(x,y):\exists z\:((x,z)\in\alpha\wedge(z,y)\in\beta)\}$. 
If $\mathcal K$ is a uniformity on $Y$ we call $(Y,{\mathcal K})$ a {\it uniform space}, also for all 
$\alpha\in{\mathcal K}$ and $x\in Y$, let $\alpha[x]:=\{y: (x,y)\in\alpha\}$; so 
$\tau:=\{U\subseteq Y:\exists\alpha\in{\mathcal K}\:(\alpha[x]\subseteq U)\}$ is a topology on 
$Y$, we call it the {\it uniform topology induced by} $\mathcal K$ and equip $(Y,{\mathcal K})$ 
by $\tau$. For topological space $Z$ we say $Z$ is {\it uniformzable} and uniformity $\mathcal H$ 
is a {\it compatible uniformity} on $Z$, if the uniform topology induced by $\mathcal H$ coincides 
with original topology of $Z$. Every compact Hausdorff topological space is uniformzable and has a 
unique compatible uniformity. For more details on uniform structures we refer the interested reader to \cite{1966}.
\\
If $A$ is a collection of maps from uniform space $(Y,{\mathcal K}_Y)$ to uniform space 
$(Z,{\mathcal K}_Z)$, we say $A$ is {\it equicontinuous} if for all $\alpha\in{\mathcal K}_Z$, 
there exists $\beta\in{\mathcal K}_Y$ with 
\linebreak
$A\beta:=\{(f(x),f(y)):f\in A,(x,y)\in\beta\}\subseteq\alpha$.
\subsection*{Background on transformation semigroups}
By a {\it (topological) transformation semigroup} $(S,Y,\pi)$ or simply $(S,Y)$ we mean a compact Hausdorff topological space $Y$ ({\it phase space}), discrete topological semigroup $S$ ({\it phase semigroup}) with identity $e$ and continuous map $\pi:S\times Y\to Y$ ($\pi(s,x)=sx$ for $s\in S$ and $x\in Y$) such that for all $x\in Y$ and $s,t\in S$ we have $ex=x$ and $(st)x=s(tx)$. Consider the transformation semigroup $(S,Y)$ with unique compatible uniformity $\mathcal H$ on $Y$, we say:
\begin{itemize}
\item $(S,Y)$ is {\it equicontinuous} if for all $\alpha\in {\mathcal H}$ 
	there exists $\beta\in {\mathcal H}$ with 
	\linebreak
	$S\beta:=\{(sx,sy):s\in S,(x,y)\in\beta\}\subseteq\alpha$;
\item $x\in Y$ is an {\it equicontinuous point of} $(S,Y)$ if for all $\alpha\in {\mathcal H}$ 
	there exists open neighbourhood $U$ of $x$ with $\{(sx,sy):s\in S,,y\in U\}\subseteq\alpha$;
\item $(S,Y)$ is {\it expansive} if there exists $\alpha\in{\mathcal H}$ such that for all 
	distinct $x,y\in Y$ there exists $s\in S$ with $(sx,sy)\notin\alpha$;
\item $(S,Y)$ is {\it sensitive} if there exists $\alpha\in{\mathcal H}$ such that for all 
	$x\in Y$ and open neighbourhood $U$ of $x$ there exists $s\in S$ and $y\in U$ with 
	$(sx,sy)\notin\alpha$. So if $(S,Y)$ is sensitive, then it has a non-equicontinuous point and it is not equicontinuous.
\end{itemize}
Moreover in the transformation semigroup $(S,Y)$ for all $s\in S$ the map $\pi^s:Y\to Y$ with $\pi^s(y)=sy$ 
(for $y\in Y$) is continuous, sometimes we denote $\pi^s$ simply by $s$. We call the 
closure of $\{\pi^s:s\in S\}$ in $Y^Y$ with pointwise convergence (or product) topology, 
the {\it enveloping semigroup} and denote it by $E(S,Y)$ or simply by $E(Y)$. If $(S,Y)$ 
is an equicontinuous transformation semigroup, then $E(S,Y)$ is an equicontinuous family 
on $Y$ \cite{1969} and \cite{1979}, in particular all elements of $E(S,Y)$ are continuous maps on  $Y$. 
The enveloping semigroup has a semigroup operation under the composition of maps \cite{1969}. 
In the transformation semigroup $(S,Y)$ we say:
\begin{itemize}
\item $(S,Y)$ is {\it distal} if it satisfies each of the following equivalent conditions 
	(hence, in distal transformation semigroup $(S,Y)$  all elements of $E(S,Y)$  
	(and in particular $S$) are bijections from $Y$ to $Y$) \cite{1969}:
	\begin{itemize}
	\item $E(S,Y)$  is a group;
	\item for all $x,y,z\in Y$ and net $\{s_\lambda\}_{\lambda\in\Lambda}$ in $S$ with 
		$\mathop{\lim}\limits_{\lambda\in\Lambda} s_\lambda x=z=\mathop{\lim}\limits_{\lambda\in\Lambda} s_\lambda y$ we have $x=y$;
	\end{itemize}
\item $(S,Y)$  is {\it weakly almost periodic}, if all elements of $E(S,Y)$  are 
	continuous maps on $Y$ (so, it's evident that all equicontinuous transformation semigroups are weakly almost periodic) \cite{1998}.
\end{itemize}
\begin{convention}\label{convention21}
In the following text suppose $X$ is a finite discrete space with at least two elements, 
$\Gamma$ is an infinite set, $\mathcal S$ is a semigroup of generalized shifts on 
$X^\Gamma$ (equip $X^\Gamma$ with product topology) containing the identity map on 
$X^\Gamma$, ${\mathcal T}=\{\varphi\in\Gamma^\Gamma:\sigma_\varphi\in{\mathcal S}\}$. 
Thus we may consider the transformation semigroup $({\mathcal S},X^\Gamma)$ where the 
elements of ${\mathcal S}$ acts in a natural way on $X^\Gamma$. For $H\subseteq \Gamma$, let:
\[\alpha_H:=\{((x_w)_{w\in\Gamma},(y_w)_{w\in\Gamma})\in X^\Gamma\times X^\Gamma:\forall w\in H\SP(x_w=y_w)\}\:.\] 
Then
\[\mathcal{F}:=\{\alpha\subseteq X^\Gamma\times X^\Gamma:{\rm \: there \: exists \: a\: finite\: subset\:}H{\rm \: of\:} \Gamma\:{\rm with}\:\alpha_H\subseteq\alpha\}\]
is the unique compatible uniformity on $X^\Gamma$.
\\
On the other hand one may verify that for all $\varphi,\eta:\Gamma\to\Gamma$, we have 
\begin{itemize}
\item $\sigma_\varphi\circ\sigma_\eta=\sigma_{\eta\circ\varphi}$;
\item $\sigma_\varphi=\sigma_\eta$ if and only if $\varphi=\eta$;
\item $\sigma_{{\rm id}_\Gamma}={\rm id}_{X^\Gamma}$ (where by ${\rm id}_Z:Z\to Z$ we mean ${\rm id}_Z(w)=w  (w\in Z)$).
\end{itemize}
Therefore, if $\mathcal M$ is a family of generalized shift on $X^\Gamma$, $\mathcal M$ is a semigroup 
(containing the identity map on $X^\Gamma$) if and only if $\{\varphi\in\Gamma^\Gamma:\sigma_\varphi\in{\mathcal M}\}$ 
is a semigroup (containing the identity map on $\Gamma$), in addition 
$\mathop{\{\varphi\in\Gamma^\Gamma:\sigma_\varphi\in{\mathcal M}\}\to{\mathcal M}}\limits_{\eta\mapsto\sigma_\eta}$ 
is bijective. In particular, $\mathcal T$ is a semigroup (under the operation of composition of maps) containing 
${\rm id}_\Gamma$ and $\mathop{{\mathcal T}\to{\mathcal S}}\limits_{\eta\mapsto\sigma_\eta}$ is bijective.
\end{convention}
\begin{note}\label{note22}
The set $\{\sigma_\varphi:\varphi\in\Gamma^\Gamma\}$ is a closed subset of $C(X^\Gamma, X^\Gamma)$ 
(the collection of all continuous maps from $ X^\Gamma$ to $ X^\Gamma$) with pointwise convergence (product) 
topology \cite{2015}, so all continuous elements of $E({\mathcal S},X^\Gamma)$ has the form $\sigma_\varphi$ for 
some $\varphi:\Gamma\to\Gamma$. In particular, $({\mathcal S},X^\Gamma)$ is weakly almost periodic if and only 
if $E({\mathcal S},X^\Gamma)\subseteq\{\sigma_\varphi:\varphi\in\Gamma^\Gamma\}$.
\\
We call $a\in A$ a {\it quasi-periodic point} (resp. {\it periodic point}) of $h:A\to A$, 
if there exist $n>m\geq1$ such that $h^n(a)=h^m(a)$ (resp. $h^m(a)=a$). For $f:A\to B$ and 
$C\subseteq A$, by $f\restriction_C:C\to B$ ($f\restriction_C(x)=f(x) (x\in C)$) we mean the restriction of $f$ to $C$.
\end{note}
\section{When is $({\mathcal S},X^\Gamma)$ equicontinuous?}
\noindent In this section we prove that $({\mathcal S},X^\Gamma)$ is equicontinuous if and only if for all $w\in\Gamma$, ${\mathcal T}w$ is finite.
\begin{lemma}\label{lemma31}
If $({\mathcal S},X^\Gamma)$ has an equicontinuous point, then for all $w\in \Gamma$, ${\mathcal T}w$ is finite.
\end{lemma}
\begin{proof}
Suppose $(q_w)_{w\in\Gamma}$ is an equicontinuous point of $({\mathcal S},X^\Gamma)$. 
Choose $v\in\Gamma$, then there exists open neighbourhood $U$ of $(q_w)_{w\in\Gamma}$ 
with $\mathcal{S}\{(z, (q_w)_{w\in\Gamma}):z\in U\}\subseteq\alpha_{\{v\}}$. There exists 
finite subset $H$ of $\Gamma$ with $\alpha_H[(q_w)_{w\in\Gamma}]\subseteq U$, hence 
for all $\varphi\in{\mathcal T}$ we have 
\[\{(\sigma_\varphi((z_w)_{w\in\Gamma}),\sigma_\varphi((q_w)_{w\in\Gamma}))
	:(z_w)_{w\in\Gamma}\in\alpha_H[(q_w)_{w\in\Gamma}]\}\subseteq\alpha_{\{v\}}\:.\]
Choose distinct $p,q\in X$ and let:
\[x_w:=\left\{\begin{array}{lc} q_w & w\in H\:, \\ p & w\in\Gamma\setminus H\:, 
	\end{array}\right.\SP{\rm and}\SP y_w:=\left\{\begin{array}{lc} q_w & w\in H\:, \\ qp & w\in\Gamma\setminus H\:. \end{array}\right.\]
Then $(x_w)_{w\in\Gamma} ,(y_w)_{w\in\Gamma}\in\alpha_H[(q_w)_{w\in\Gamma}]$, therefore
\[(\sigma_\varphi((x_w)_{w\in\Gamma}),\sigma_\varphi((q_w)_{w\in\Gamma})),
(\sigma_\varphi((y_w)_{w\in\Gamma}),\sigma_\varphi((q_w)_{w\in\Gamma}))\in\alpha_{\{v\}}\:.\]
Thus $x_{\varphi(v)}=q_{\varphi(v)}=y_{\varphi(v)}$ which leads to $\varphi(v)\in H$, $\varphi$ 
is an arbitrary element of $\mathcal T$, we have $\mathcal{T}v\subseteq H$. Since $H$ is finite, $\mathcal{T}v$ is finite too.
\end{proof}
\begin{lemma}\label{lemma32}
If for all $w\in\Gamma$, $\mathcal{T}w$ is finite, then $({\mathcal S},X^\Gamma)$  is equicontinuous.
\end{lemma}
\begin{proof}
Suppose for all $w\in\Gamma$, $\mathcal{T}w$  is finite and $\alpha\in{\mathcal F}$. 
There exist $\beta_1,\ldots,\beta_m\in\Gamma$ with $\alpha_{\{\beta_1,\ldots,\beta_m\}}\subseteq\alpha$. 
Thus $H:=\mathcal{T}\{\beta_1,\ldots,\beta_m \}=\{\varphi(\beta_i):\varphi\in{\mathcal T},1\leq i\leq m\}$ 
is a finite subset of $\Gamma$ and $\alpha_H\in{\mathcal F}$. We show 
${\mathcal S}\alpha_H \subseteq\alpha$, for this aim suppose 
$((x_w)_{w\in\Gamma} ,(y_w)_{w\in\Gamma})\in\alpha_H$ and consider:

$((x_w)_{w\in\Gamma} ,(y_w)_{w\in\Gamma})\in\alpha_H$
\begin{eqnarray*}
& \Rightarrow & (\forall w\in H\SP(x_w=y_w)) \\
& \Rightarrow & (\forall\varphi\in\mathcal{T}\SP\forall i\in\{1,\ldots,m\}\SP(x_{\varphi(\beta_i)}=y_{\varphi(\beta_i)})) \\
& \Rightarrow & (\forall\varphi\in\mathcal{T}\SP(((x_{\varphi(w)})_{w\in\Gamma} ,(y_{\varphi(w)})_{w\in\Gamma})\in\alpha_{\{\beta_1,\ldots,\beta_m\}}) \\
& \Rightarrow & (\forall\varphi\in\mathcal{T}\SP((\sigma_\varphi((x_w)_{w\in\Gamma} ),\sigma_\varphi((y_w)_{w\in\Gamma}))\in\alpha_{\{\beta_1,\ldots,\beta_m\}}) \\
& \Rightarrow & (\forall\sigma_\varphi\in\mathcal{S}\SP((\sigma_\varphi((x_w)_{w\in\Gamma} ),\sigma_\varphi((y_w)_{w\in\Gamma}))\in\alpha_{\{\beta_1,\ldots,\beta_m\}}) \\
& \Rightarrow & (\forall\sigma_\varphi\in\mathcal{S}\SP((\sigma_\varphi((x_w)_{w\in\Gamma} ),\sigma_\varphi((y_w)_{w\in\Gamma}))\in\alpha) 
\end{eqnarray*}
which leads us to the desired result.
\end{proof}
\begin{lemma}\label{lemma33}
If $({\mathcal S},X^\Gamma)$ is weakly almost periodic, then for all $w\in\Gamma$, ${\mathcal T}w$ is finite.
\end{lemma}
\begin{proof}
Suppose $({\mathcal S},X^\Gamma)$ is weakly almost periodic and consider 
$w_0\in\Gamma$ such that ${\mathcal T}w_0$ is infinite, choose $\beta_1,\beta_2,\ldots\in{\mathcal T}$ 
such that $\{\beta_n(w_0)\}_{n\geq1}$ is a one--to--one sequence. We may suppose $X$ has a finite 
cyclic group structure with identity $u$. Choose $v\in X\setminus\{u\}$. For $i\geq1$ 
choose $y_i=(y^i_w)_{w\in\Gamma}\in X^\Gamma$ and $y=(y_w)_{w\in\Gamma}\in X^\Gamma$ such that:
\[y_w^i:=\left\{\begin{array}{lc} v & w=\beta_i(w_0)\:, \\ u & {\rm otherwise}\:, \end{array}\right.\SP
	{\rm and}\SP y_w:=\left\{\begin{array}{lc} v & w=\beta_1(w_0),\beta_2(w_0),\ldots\:, \\ u & {\rm otherwise}\:. \end{array}\right.\]
Then $\mathop{\lim}\limits_{i\to\infty}(y_1+y_2+\cdots+y_i)=y$. The sequence 
$\{\sigma_{\beta_i}\}_{i\geq1}$ has a convergent subnet $\{\sigma_{\beta_{i_\lambda}}\}_{\lambda\in\Lambda}$ 
in $ E({\mathcal S},X^\Gamma)$ to $p\in E({\mathcal S},X^\Gamma)$. Since $({\mathcal S},X^\Gamma)$ is 
weakly almost periodic, by Note~\ref{note22}, there exists $\psi:\Gamma\to\Gamma$ with $p=\sigma_\psi$. 
Let $(z_w)_{w\in\Gamma}:=py=(y_{\psi(w)})_{w\in\Gamma}$ and 
$(z^n_w)_{w\in\Gamma}:=py_n=(y^n_{\psi(w)})_{w\in\Gamma}$ ($n\geq1$). Using the continuity of $p$ 
and 
\linebreak
$\mathop{\lim}\limits_{i\to\infty}(y_1+y_2+\cdots+y_i)=y$ we have:
\[\mathop{\lim}\limits_{i\to\infty}(py_1+\cdots+py_i)= \mathop{\lim}\limits_{i\to\infty}p(y_1+\cdots+y_i)=py\:,\]
which leads to:
\[\mathop{\lim}\limits_{i\to\infty}
(z^1_{w_0}+ z^2_{w_0}+\cdots+ z^i_{w_0})= 
z_{w_0}\:.\tag{*}\]
For $n\geq1$ and all $i\geq n$ we have $y^n_{\beta_i(w_0)}=u$, so 
$\mathop{\lim}\limits_{i\to\infty} y^n_{\beta_i(w_0)}=u$, which leads to 
$z^n_{w_0}=\mathop{\lim}\limits_{\lambda\in\Lambda} y^n_{\beta_{i_\lambda}(w_0)}=u$. Hence: 
\[\forall n\geq1\SP(z^1_{w_0}+ z^2_{w_0}+\cdots+ z^n_{w_0}=u)\:\tag{**}\]
Using (*) and (**) we have $z_{w_0}=u$. On the other hand for all $i\geq1$ we have 
$y_{\beta_i(w_0)}=v$, therefore $v=\mathop{\lim}\limits_{i\to\infty} y_{\beta_i(w_0)}= \mathop{\lim}\limits_{\lambda\in\Lambda} y_{\beta_{i_\lambda}(w_0)}=z_{w_0}$, 
which is a contradiction by $u\neq v$.
\end{proof}
\begin{lemma}\label{lemma34}
If $({\mathcal S},X^\Gamma)$  is not equicontinuous, if and only if it is sensitive.
\end{lemma}
\begin{proof}
Suppose $({\mathcal S},X^\Gamma)$  is not equicontinuous, by Lemma~\ref{lemma32} there exists 
$v\in\Gamma$ such that ${\mathcal T}v$ is infinite. Consider $x=(x_w)_{w\in\Gamma}\in X^\Gamma$ 
and open neighbourhood $U$ of $x$, there exists $\beta_1,\ldots,\beta_n\in\Gamma$ such that 
\[\{(y_w)_{w\in\Gamma}\in X^\Gamma:y_{\beta_1}=x_{\beta_1},\ldots, y_{\beta_n}=x_{\beta_n}\}\subseteq U\:.\] 
Choose $\beta\in{\mathcal T}v\setminus\{\beta_1,\ldots,\beta_n\}$ and $p\in X\setminus\{x_\beta\}$ also let:
\[y_w:=\left\{\begin{array}{lc} x_w & w\neq\beta\:, \\ p & w=\beta\: , \end{array}\right.\] 
then $y:= (y_w)_{w\in\Gamma}\in U$. There exists $\varphi\in{\mathcal T}$ with $\beta=\varphi(v)$, 
so for $(u_w)_{w\in\Gamma}:=\sigma_\varphi(x)$ 
and $(t_w)_{w\in\Gamma}:=\sigma_\varphi(y)$  we have $u_v=x_\beta\neq p=y_\beta=t_v$, thus 
\[(\sigma_\varphi(x),\sigma_\varphi(y))=( (u_w)_{w\in\Gamma},(t_w)_{w\in\Gamma}\notin\alpha_{\{v\}}\:.\]
Hence for all $x\in X^\Gamma$ and open neighbourhood $U$ of $x$ there exists $y\in U$ 
and $\sigma_\varphi\in{\mathcal S}$ with $(\sigma_\varphi(x),\sigma_\varphi(y)) \notin\alpha_{\{v\}}$ which completes the proof.
\end{proof}
\begin{theorem}[main]\label{theorem35}
In the transformation semigroup $({\mathcal S},X^\Gamma)$, the following statements are equivalent:
\begin{itemize}
\item[1.] the transformation semigroup $({\mathcal S},X^\Gamma)$ is equicontinuous;
\item[2.] transformation semigroup $({\mathcal S},X^\Gamma)$ has at least an equicontinuous point;
\item[3.] the transformation semigroup $({\mathcal S},X^\Gamma)$ is not sensitive;
\item[4.] all of the elements of $E({\mathcal S},X^\Gamma)$ are continuous maps on $X^\Gamma$;
\item[5.] for all $w\in\Gamma$, $\{\varphi(w):\sigma_\varphi\in{\mathcal S}\}$ is finite;
\item[6.] for all $w\in\Gamma$, $\{\varphi(w):\sigma_\varphi\in E({\mathcal S},X^\Gamma)\}$ is finite;
\item[7.] $E({\mathcal S},X^\Gamma)\subseteq\{\sigma_\varphi:\varphi\in\Gamma^\Gamma\}$.
\end{itemize}
\end{theorem}
\begin{proof}
(1,2,3,4,5) are equivalent: First note that if $(S,Z)$ is an equicontinuous transformation semigroup, then all 
of its points are equicontinuous points, it is not sensitive, and all of the elements of $E(S,Z)$ are continuous, 
so (1) implies (2), (3) and (4). By Lemmas~\ref{lemma31} and~\ref{lemma33}, (2) and (4) imply (5). By Lemma~\ref{lemma32}, (5) implies 
(1). By Lemma~\ref{lemma34}, (3) implies (1).
\\
(4) and (7) are equivalent by Note~\ref{note22}.
\\
It’s evident that (6) implies (5).
\\
Finally, note that in equicontinuous transformation semigroup $(S,Z)$, one may 
consider the  equicontinuous transformation semigroup $(E(S,Z),Z)$, now using the 
equivalency of (1) and (5), two items (1) and (7) imply (6).
\end{proof}
\section{When is $({\mathcal S},X^\Gamma)$ distal?}
\noindent In this section we prove $({\mathcal S},X^\Gamma)$ is distal if and only if it is equicontinuous, 
and $\mathcal T$  is a collection of permutations on $\Gamma$. Let’s recall that $\varphi:\Gamma\to\Gamma$ 
is injective (resp. surjective) if and only if $\sigma_\varphi:X^\Gamma\to X^\Gamma$ 
is injective (resp. surjective) \cite{2008}.
\begin{lemma}\label{lemma41}
If $({\mathcal S},X^\Gamma)$ is distal, then for all $w\in\Gamma$ the set 
$\mathcal{T}^{-1}w(=\{t\in\Gamma:\exists\varphi\in\Gamma\:(\varphi(t)=w)\})$ 
is finite and for all $\varphi\in{\mathcal T}$ the map $\varphi:\Gamma\to\Gamma$ is bijective.
\end{lemma}
\begin{proof}
If $({\mathcal S},X^\Gamma)$ is distal, then for all $\varphi\in{\mathcal T}$  the map 
$\sigma_\varphi:X^\Gamma\to X^\Gamma$ is bijective, so $\varphi:\Gamma\to\Gamma$  
is bijective too. Consider $v\in\Gamma$. If ${\mathcal T}^{-1}v$ is infinite, then choose the 
one to one sequence $\{v_n\}_{n\geq1}$ in ${\mathcal T}^{-1}v$ and 
sequence $\{\varphi_n\}_{n\geq1}$ in $\mathcal T$ such that for all $n\geq1$ we 
have $\varphi_n(v_n)=v$ (thus $\{\varphi_n\}_{n\geq1}$ is a one-to-one sequence too). 
Choose distinct $p,q\in X$ and let:
\[z_w:=\left\{\begin{array}{lc} p & w=v\:,\\ q & w\in\Gamma\setminus\{v\}\:.\end{array}\right.\] 
It's evident that 
$\mathop{\lim}\limits_{n\to\infty}\sigma_{\varphi_n}((q)_{w\in\Gamma})=\mathop{\lim}\limits_{n\to\infty}(q)_{w\in\Gamma}=(q)_{w\in\Gamma}$. 
We prove 
\linebreak
$\mathop{\lim}\limits_{n\to\infty}\sigma_{\varphi_n}((z_w)_{w\in\Gamma})=(q)_{w\in\Gamma}$. 
So we should prove $\mathop{\lim}\limits_{n\to\infty}z_{\varphi_n(w)}=q$ for all $w\in\Gamma$. 
Consider $w\in\Gamma$ there exists $N\geq1$ such that $w\neq v_n$ for all $n\geq N$, 
thus $\varphi_n(w)\neq\varphi_n(v_n)$ which leads to $\varphi_n(w)\neq v$ and $z_{\varphi_n(w)}=q$ 
for all $n\geq N$, hence $\mathop{\lim}\limits_{n\to\infty}z_{\varphi_n(w)}=q$ which leads to the desired result. 
Since 
$\mathop{\lim}\limits_{n\to\infty}\sigma_{\varphi_n}((z_w)_{w\in\Gamma})=\mathop{\lim}\limits_{n\to\infty}\sigma_{\varphi_n}((q)_{w\in\Gamma})$ 
and $(z_w)_{w\in\Gamma}\neq (q)_{w\in\Gamma}$, $({\mathcal S},X^\Gamma)$ is not distal, 
which is in contradiction with our assumption and completes the proof.
\end{proof}
\begin{lemma}\label{lemma42}
Suppose $\mathcal T$ is a semigroup of permutations on $\Gamma$ and $w\in\Gamma$. The following statements are equivalents:
\begin{itemize}
\item[1.] ${\mathcal T}w$ is finite;
\item[2.] ${\mathcal T}^{-1}w$  is finite;
\item[3.] ${\mathcal T}w={\mathcal T}^{-1}w $  is finite.
\end{itemize}
\end{lemma}
\begin{proof}
(1) $\Rightarrow$ (2): Suppose ${\mathcal T}w$ is finite and consider $\varphi\in{\mathcal T}$, since 
$\{\varphi^n(w):n\geq1\}\subseteq{\mathcal T}w$, the set $\{\varphi^n(w):n\geq1\}$ is finite too, so 
there exists $n>m$ with $\varphi^n(w)=\varphi^m(w)$, hence $\varphi^{n-m}(w)=w$ and $w$ is a 
periodic point of $\varphi$, however $\varphi^{2(n-m)-1}\in{\mathcal T}$ which leads to 
$\varphi^{-1}(w)=\varphi^{2(n-m)-1}(w)\in{\mathcal T}w$. Thus ${\mathcal T}^{-1}w\subseteq{\mathcal T}w$ 
and ${\mathcal T}^{-1}w$ is finite.
\\
(2) $\Rightarrow$ (1): Suppose ${\mathcal T}^{-1}w$ is finite, then 
$\mathcal{T}^\prime:=\{\varphi^{-1}:\varphi\in{\mathcal T}\}(={\mathcal T}^{-1})$ is a semigroup of 
permutations on $\Gamma$ and ${\mathcal T}^\prime w$ is finite, now using ``(1) $\Rightarrow$  (2)'', 
$\mathcal{T}^{\prime\:{-1}}w=\mathcal{T}w$ is finite.
\\
Using the proof of ``(1) $\Rightarrow$  (2)'', and the above proof it’s clear that ``(1,2) $\Rightarrow$  (3)''.
\end{proof}
\begin{lemma}\label{lemma43}
If for all $\varphi\in{\mathcal T}$, $\varphi:\Gamma\to\Gamma$ is bijective and for 
all $w\in\Gamma$, $\mathcal{T}w={\mathcal T}^{-1}w$ is finite, then $({\mathcal S},X^\Gamma)$ is distal.
\end{lemma}
\begin{proof}
Suppose for all $\varphi\in{\mathcal T}$, $\varphi:\Gamma\to\Gamma$ is bijective and for all 
$w\in\Gamma$, ${\mathcal T}w={\mathcal T}^{-1}w$ is finite. Consider the net 
$\{\varphi_\lambda\}_{\lambda\in\Lambda}$ in ${\mathcal T}$ and 
$(x_w)_{w\in\Gamma},(y_w)_{w\in\Gamma},(z_w)_{w\in\Gamma}\in X^\Gamma$ 
with $\mathop{\lim}\limits_{\lambda\in\Lambda}\sigma_{\varphi_\lambda}((x_w)_{w\in\Gamma})=(z_w)_{w\in\Gamma}=\mathop{\lim}\limits_{\lambda\in\Lambda}\sigma_{\varphi_\lambda}((y_w)_{w\in\Gamma})$. 
Choose $v\in\Gamma$, the map $\varphi\restriction_{{\mathcal T}v}:{\mathcal T}v\to{\mathcal T}v$ is 
bijective since it is one to one and ${\mathcal T}v$ is finite. We have $\mathop{\lim}\limits_{\lambda\in\Lambda}\sigma_{\varphi_\lambda}\restriction_{X^{\mathcal{T}v}}((x_w)_{w\in\mathcal{T}v })=(z_w)_{w\in\mathcal{T}v }=\mathop{\lim}\limits_{\lambda\in\Lambda}\sigma_{\varphi_\lambda}\restriction_{X^{\mathcal{T}v}} ((y_w)_{w\in\mathcal{T}v })$, 
note to the fact that $\{\sigma_{\varphi_\lambda}\restriction_{X^{\mathcal{T}v}}:\lambda\in\Lambda\}$ is a 
set of permutations on $X^{{\mathcal T}v}$ and it is finite (since $X^{{\mathcal T}v}$ is finite), leads us to 
$(x_w)_{w\in{\mathcal T}v }=(y_w)_{w\in{\mathcal T}v }$ and in particular $x_v=y_v$. 
Hence $(x_w)_{w\in\Gamma}=(y_w)_{w\in\Gamma}$ and $({\mathcal S},X^\Gamma)$ is distal.
\end{proof}
\noindent Using Theorem~\ref{theorem35}, Lemmas~\ref{lemma41},~\ref{lemma42} and~\ref{lemma43} we have the following theorem.
\begin{theorem}\label{theorem45}
In the transformation semigroup $({\mathcal S},X^\Gamma)$, the following statements are equivalent:
\begin{itemize}
\item the transformation semigroup $({\mathcal S},X^\Gamma)$ is equicontinuous and for all $\varphi\in\mathcal{T}$, $\varphi:\Gamma\to\Gamma$ is bijective;
\item the transformation semigroup $({\mathcal S},X^\Gamma)$ is distal;
\item for all $w\in\Gamma$, $\{\varphi(w):\sigma_\varphi\in{\mathcal S}\}$ is finite and for all $\varphi\in\mathcal{T}$, $\varphi:\Gamma\to\Gamma$ is bijective;
\item for all $w\in\Gamma$, $\{v\in\Gamma:\exists\sigma_\varphi\in{\mathcal S}\:(\varphi(v)=w)\}$ is finite and for all $\varphi\in\mathcal{T}$, $\varphi:\Gamma\to\Gamma$ is bijective.
\end{itemize}
\end{theorem}
\section{Expansive generalized shift transformation semigroup $({\mathcal S},X^\Gamma)$}
\noindent In this section we prove that the transformation semigroup $({\mathcal S},X^\Gamma)$ 
is expansive, if and only if there exists finite subset $H$ of $\Gamma$ such that $\Gamma={\mathcal T}H$, 
in particular using Theorem\ref{theorem35} one may verify directly that if $({\mathcal S},X^\Gamma)$ is expansive, 
then it is sensitive.
\begin{theorem}\label{theorem51}
The transformation semigroup $({\mathcal S},X^\Gamma)$ is expansive, if and 
only if there exists finite subset $H$ of $\Gamma$ such that $\Gamma={\mathcal T}H$.
\end{theorem}
\begin{proof}
Suppose $({\mathcal S},X^\Gamma)$ is expansive, there exists 
$\alpha\in{\mathcal F}$ such that for all distinct $x,y\in X^\Gamma$ there exists 
$\varphi\in{\mathcal T}$ with $(\sigma_\varphi(x),\sigma_\varphi(y))\notin\alpha$. 
There exists finite subset $H$ of $\Gamma$ with $\alpha_H\subseteq\alpha$. 
Choose $v\in\Gamma$ and distinct $p,q\in X$ and let:
\[z_w:=\left\{\begin{array}{lc} p & w=v\:, \\ q & w\in\Gamma\setminus\{v\}\:.\end{array}\right.\]
There exists $\varphi\in{\mathcal T}$ with 
$((z_{\varphi(w)})_{w\in\Gamma},(q)_{w\in\Gamma})=(\sigma_\varphi(z_w)_{w\in\Gamma}),\sigma_\varphi((q)_{w\in\Gamma}))\notin\alpha$ 
which leads to $((z_{\varphi(w)})_{w\in\Gamma},(q)_{w\in\Gamma})\notin\alpha_H$, 
thus there exists $w\in H$ with $z_{\varphi(w)}\neq q$, and $\varphi(w)=v$, 
hence $v\in\varphi(H)$ and $\Gamma=\mathcal{T}H$.
\\
Now conversely, suppose there exists finite subset $H$ of $\Gamma$ 
such that $\Gamma=\mathcal{T}H$ also consider distinct 
$(x_v)_{v\in\Gamma},(y_v)_{v\in\Gamma}\in X^\Gamma$. There exists 
$w\in\Gamma$ with $x_w\neq y_w$. Also there exist $\varphi\in{\mathcal T}$ 
and $h\in H$ with $\varphi(h)=w$, thus $x_{\varphi(h)}\neq y_{\varphi(h)}$ and 
$((x_{\varphi(v)})_{v\in\Gamma},(y_{\varphi(v)})_{v\in\Gamma} )=(\sigma_\varphi((x_v)_{v\in\Gamma}),\sigma_\varphi((y_v)_{v\in\Gamma}) )\notin\alpha_H$.
\end{proof}
\section{A diagram}
\noindent Now we are ready to compare equicontinuity, distality, sensitivity and expansivity in generalized 
shift transformation semigroups, via diagrams and examples. For this aim suppose $\mathcal C$ is the
 collection of all transformation semigroups like $({\mathcal E},Y^\Lambda)$ where $Y$ is a finite 
 discrete space with at least two elements, $\Lambda$ is a nonempty set and $\mathcal E$ is a 
 subsemigroup of generalized shifts on $Y^\Lambda$ (i.e., $\mathcal E$ is a subsemigroup of 
 $\{\sigma_\varphi:\varphi\in\Lambda^\Lambda\}$), then we have the following diagram:
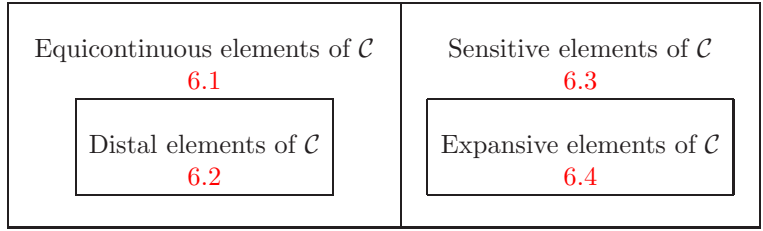
\begin{center}
\begin{tabular}{|c|c|}
\hline & \\
\begin{tabular}{c}
Equicontinuous elements of $\mathcal C$ \\
\ref{example1} \\
\begin{tabular}{|c|}
\hline \\
Distal elements of $\mathcal C$ \\
\ref{example2} \\
\hline 
\end{tabular} \\
\end{tabular}
&
\begin{tabular}{c}
Sensitive elements of $\mathcal C$ \\
\ref{example3} \\
\begin{tabular}{|c|}
\hline \\
Expansive elements of $\mathcal C$ \\
\ref{example4} \\
\hline 
\end{tabular} \\
\end{tabular}
\\ &
\\ \hline 
\end{tabular}
\end{center}
 \begin{example}\label{example1}
 Consider $\varphi:{\mathbb Z}\to{\mathbb Z}$ with $\varphi(n)=|n|$ (for $n\in{\mathbb Z}$), 
 then 
 \linebreak
 $(\{\sigma_\varphi,\sigma_{{\rm id}_Z}\},\{0,1\}^{\mathbb Z})$ is equicontinuous and it is not distal 
 (note that $\sigma_\varphi:\{0,1\}^{\mathbb Z}\to\{0,1\}^{\mathbb Z}$ is not surjective).
 \end{example}
 \begin{example}\label{example2}
 Consider $\varphi:{\mathbb Z}\to{\mathbb Z}$ with $\varphi(n)=-n$ (for $n\in{\mathbb Z}$), 
 then 
 \linebreak
 $(\{\sigma_\varphi,\sigma_{{\rm id}_Z}\},\{0,1\}^{\mathbb Z})$  is distal.
 \end{example}
 \begin{example}\label{example3}
 Consider $\varphi:{\mathbb Z}\to{\mathbb Z}$ with $\varphi(n)=n^2$ (for $n\in{\mathbb Z}$), 
 then 
 \linebreak
 $(\{\sigma_{\varphi^n}:n\geq0\},\{0,1\}^{\mathbb Z})$  is sensitive and it is not expansive 
 (since $\{\varphi^n(2):n\geq0\}=\{2^{2^n}:n\geq0\}$ is infinite, $\{\varphi^n(1):n\geq0\}=\{1\}$ is 
 finite, and for all finite subset $A$ of $\mathbb Z$ the set $\{\varphi^n(i):i\in A,n\geq0\}=\{i^{2^n}:i\in A,n\geq0\}$ 
 is a proper subset of $\mathbb Z$). 
 \end{example}
 \begin{example}\label{example4}
 Consider $\varphi:{\mathbb Z}\to{\mathbb Z}$ with:
\[\varphi(n):=\left\{\begin{array}{lc} n+1 & n\geq1 \:, \\ 0 & n=0 \:, \\ n-1 & n\leq-1\:, \end{array}\right.\] 
then $(\{\sigma_{\varphi^n}:n\geq0\},\{0,1\}^{\mathbb Z})$ is expansive.
 \end{example}

\noindent
{\small {\bf Fatemah Ayatollah Zadeh Shirazi},
Faculty of Mathematics, Statistics and Computer Science,
College of Science, University of Tehran ,
Enghelab Ave., Tehran, Iran
\\
({\it e-mail}: fatemah@khayam.ut.ac.ir)
\\
{\bf Fatemeh Ebrahimifar},
Faculty of Mathematics, Statistics and Computer Science,
College of Science, University of Tehran ,
Enghelab Ave., Tehran, Iran
\\
({\it e-mail}: ebrahimifar64@ut.ac.ir)}

\end{document}